\documentclass[12pt, reqno]{amsart}
\usepackage{amsfonts}
\usepackage{bbm}
\usepackage{amscd,amsfonts}
\usepackage{amssymb, eucal, amsfonts, amsmath, xypic, latexsym, tikz}
\usepackage{pifont}
\usepackage{mathrsfs,color}
\usepackage{amsthm,indentfirst,bm,fancyhdr,dsfont}
\usepackage{graphicx}
\usepackage[all]{xy}
\usepackage[CJKbookmarks=true]{hyperref}
\usepackage{extarrows}

\usepackage{mathrsfs}
\usepackage{amsmath}
\usepackage{amssymb}
\usepackage{hyperref}

\include{micro}

\newtheorem{thm}{Theorem}[section]
\newtheorem{lemma}[thm]{Lemma}
\newtheorem{remark}[thm]{Remark}

\theoremstyle{definition}

\newtheorem{proposition}[thm]{Proposition}

 \theoremstyle{remark}

\numberwithin{equation}{section}

\begin{document}

\def\frakl{{\mathfrak L}}
\def\frakg{{\mathfrak G}}
\def\bbf{{\mathbb F}}
\def\bbl{{\mathbb L}}
\def\bbz{{\mathbb Z}}
\def\bbr{{\mathbb R}}
\def\bbc{{\mathbb C}}

\def\bvp{\bf{\varphi}}

\def\ad{\textsf{ad}}
\def\GL{\text{GL}}
\def\Der{\mbox{\rm Der}}
\def\Hom{\textsf{Hom}}
\def\ind{\textsf{ind}}
\def\res{\textsf{res}}
\def\gl{\frak{gl}}
\def\sl{\frak{sl}}
\def\Ker{\textsf{Ker}}
\def\Lie{\textsf{Lie}}
\def\id{\textsf{id}}
\def\det{\textsf{det}}
\def\Lie{\textsf{Lie}}
\def\Aut{\textsf{Aut}}
\def\Ext{\textsf{Ext}}
\def\Coker{\textsf{{Coker}}}
\def\dim{\textsf{{dim}}}
\def\pr{\mbox{\sf pr}}
\def\SL{\text{SL}}

\def\geqs{\geqslant}

\def\ba{{\mathbf a}}
\def\bd{{\mathbf d}}
\def\bbk{{\mathbb K}}
\def\co{{\mathcal O}}
\def\cn{{\mathcal N}}
\def\cv{{\mathcal V}}
\def\cz{{\mathcal Z}}
\def\cq{{\mathcal Q}}
\def\cf{{\mathcal F}}
\def\cc{{\mathcal C}}
\def\ca{{\mathcal A}}

\def\sl{{\frak{sl}}}

\def\ggg{{\frak g}}
\def\lll{{\frak l}}
\def\hhh{{\frak h}}
\def\nnn{{\frak n}}
\def\sss{{\frak s}}
\def\bbb{{\frak b}}
\def\ccc{{\frak c}}
\def\ooo{{\mathfrak o}}
\def\ppp{{\mathfrak p}}
\def\uuu{{\mathfrak u}}

\def\p{{[p]}}
\def\modf{\text{{\bf mod}$^F$-}}
\def\modr{\text{{\bf mod}$^r$-}}

\title[Invariants and dualities of parabolic groups] {Invariants and dualities of a certain parabolic group}
\author{Bin Liu}
\address{School of Mathematical Sciences, East China Normal University, Shanghai, 200241, China.} \email{1918724868@qq.com}
\subjclass[]{}
 \keywords{enhanced group, Schur-Weyl duality, degenerate double Hecke algebra}
\thanks{This work is partially supported by the NSF of China (Grant: 12071136).}

\begin{abstract} In this note, I will prove a conjecture in \cite{BYY}, which is related to the invariants of a maximal parabolic subgroup of $\GL_{n+1}$. Consequently, the natural tensor invariants of this typical maximal parabolic subgroup of $\GL_{n+1}$ are determined when $n\geq r$.
\end{abstract}

\maketitle
\section*{0. Introduction}
The purpose of this note is to prove Conjecture \eqref{eq: 0.1} proposed in \cite{BYY}, which is actually much related to tensor invariants of an distinguished maximal parabolic subgroup of $\GL_{n+1}$. The authors in \cite{BYY} defined an enhanced reductive algebraic group $\underline{G}$, which is naturally a semi-reductive group. In order to illustrate the tensor invariants of $\underline{\GL_n}=\GL_n\ltimes \mathbb{C}^n$ in $\text{End}_{\mathbb{C}}(\underline{V}^{\otimes r})$ with $\underline{V}=\mathbb{C}^{n+1}$, they introduced a finite-dimensional degenerate double Hecke algebra $D(n,r)$, and got a Levi Schur-Weyl duality for $\underline{V}=\mathbb{C}^{n+1}$ in \cite{BYY} (see \cite[Theorem 4.3]{BYY}) which is an analog of classical Schur-Weyl duality.

By \cite[Example 5.2]{BYY}, $\mathbb{C}\Psi(\frak S_r) \subsetneq \text{End}_{\underline{\GL_n} \rtimes \textbf{G}_m}(\underline{V}^{\otimes r})$ when $n < r$. However in the case $n \geq r$, there is a conjecture 
\begin{equation}\label{eq: 0.1}
\text{End}_{\underline{\GL_n} \rtimes \textbf{G}_m}(\underline{V}^{\otimes r}) =\mathbb{C}\Psi(\frak S_r)
\end{equation}
  (see \cite[Conjecture 5.4]{BYY}). In \S\ref{2} of this article, I will give a proof of \eqref{eq: 0.1}. Then it is a position to
investigate the enhanced tensor invariants, and this is listed in \S\ref{3}.

Note that $\underline{\GL_n} \rtimes \textbf{G}_m$ is a canonical maximal parabolic subgroup of $\GL_{n+1}$. It is worth mentioning that the study of invariants beyond reductive groups is a challenge (see \cite{G1}-\cite{G3}). Consequently the invariant property of $\underline{G} \rtimes \textbf{G}_m$ has its own interest.

\section{Levi Schur-Weyl duality.}\label{1}

\subsection{A general setting-up.} In this note, all vector spaces are over a complex number field of $\mathbb{C}$. Let $G$ be a connected reductive algebraic group over $\mathbb{C}$, and $(V,\rho)$  be a finite-dimensional rational representation of $G$ with representation space $V$ over $\mathbb{C}$. We can define an enhanced reductive algebraic group $\underline{G}=G \times_{\rho} V$ as follows (see \cite{BYY} explicitly):

Regard $\underline{G}$ as a set, we have $\underline{G}=G \times V$; For any $(g_1, v_1), (g_2, v_2) \in G \times V$,
\begin{align}\label{def of group}
(g_1, v_1) \cdot (g_2, v_2):=(g_1g_2, \rho(g_1)(v_2)+v_1).
\end{align}

From now on, we will write down $e^v$ for $(e,v)$ (where $e \in G$ is the identity), and always suppose that $\dim V=n$, i.e. $V \cong \mathbb{C}^n$. Let $G_1$ and $G_2$ be subgroups of $G$, we denote by $\langle G_1,G_2\rangle$ the subgroup of $G$ generated by $G_1$ and $G_2$. All representations for algebraic groups are always assumed to be rational.

In particular, assume that $G$ is a closed subgroup of $\GL(V)$. Since $\GL(V)$ can naturally act on $V$, we can defined the enhanced reductive algebraic group $\underline{G}$ as above (where $\rho$ is just the natural representation of $G$ on $V$). Note that $V$ is a closed subgroup of $\underline{G}$, then an irreducible $G$-module becomes naturally an irreducible module of $\underline{G}$ with trivial $V$-action. The isomorphism classes of irreducible rational representations of $\underline{G}$ coincide with the ones of $G$.

Let $\underline{V}:=V \oplus \mathbb{C}\eta \cong \mathbb{C}^{n+1}$, which is called the enhanced space of $V$. Naturally $\underline{V}$ becomes a $\underline{G}$-module that is defined for any $\underline{g}=(g,v) \in \underline{G}$ and $\underline{u}=u+a\eta \in \underline{V}$ with $g \in G$, $u,v \in V$ and $a \in \mathbb{C}$, via
\begin{align}\label{def of module}
\underline{g} \cdot \underline{u}:=\rho(g)(u)+av+a\eta.
\end{align}
It is easy to see that this module is a rational module of $\underline{G}$.

In particular, if $G$ is a classical group, then $G \subseteq \GL(V)$. Since $\GL(V)$ can naturally act on $V$, we can defined $\underline{G}$ and $\underline{G}$-module $\underline{V}$ as above.

Since $\GL(V) \cong \GL_n$ and $\GL(\underline{V}) \cong \GL_{n+1}$, we can regard $\GL(V)$ and $\GL(\underline{V})$ as $\GL_n$ and $\GL_{n+1}$ respectively. There are canonical imbeddings of algebraic groups $\GL_n \hookrightarrow \GL_{n+1}$ given by 
\begin{align}
g \mapsto \text{diag}(g,1):=\left(\begin{matrix}
g & 0 \\
0 & 1
\end{matrix}\right), \forall g \in \GL_n.
\end{align}
and $\textbf{G}_m \hookrightarrow \GL_{n+1}$ given by \begin{align}
c \mapsto \text{diag}(1,\cdots,1,c):=\left(\begin{matrix}
I_n \\
 & c \\
\end{matrix}\right), \forall c \in \textbf{G}_m=\mathbb{C}^{\times}.
\end{align}
where $I_n$ is $n$ identity matrix and $\textbf{G}_m$ is the multiplicative algebraic group. In this sense, both $\GL_n$ and $\textbf{G}_m$ can be viewed as closed subgroups of $\GL_{n+1}$. 

Obviously, the closed subgroups $\langle \GL(V),\textbf{G}_m\rangle$ and $\langle \underline{\GL(V)},\textbf{G}_m\rangle$ are isomorphic to $\GL(V) \times {\textbf{G}}_m$ and $\underline{\GL(V)} \rtimes {\textbf{G}}_m$ respectively. We identify all throughout the note.

Note that $\underline{\GL(V)} \rtimes {\textbf{G}}_m$ is actually a parabolic subgroup of $\GL(\underline{V})$ associated with the Levi subgroup $\GL(V) \times {\textbf{G}}_m$. Clearly, $\underline{\GL(V)}$ is a closed subgroup of $\underline{\GL(V)} \rtimes {\textbf{G}}_m$, then $e^v \in \underline{\GL(V)} \rtimes \textbf{G}_m$ for all $v \in V$.

The fact that $\GL(\underline{V})$ can naturally act on $\underline{V}$ implies that $\underline{V}$ can be viewed as a natural $\GL(\underline{V})$-module. Note that $\GL(V), \underline{\GL(V)}, \GL(V) \times \textbf{G}_m$ and $\underline{\GL(V)} \rtimes \textbf{G}_m$ are all closed subgroups of $\GL(\underline{V})$, hence $\underline{V}$ can be viewed as a natural $\GL({V})$-module, $\underline{\GL(V)}$-module, $\GL(V) \times \textbf{G}_m$-module and $\underline{\GL(V)} \rtimes \textbf{G}_m$-module respectively. It is easy to see that the natural $\underline{\GL(V)}$-module $\underline{V}$ coincides with (\ref{def of module}), where $\rho$ is just the natural representation of $\GL(V)$ on $V$.

\subsection{Classical Schur-Weyl duality.}
We can define a representation $(\underline{V}^{\otimes r},\Phi)$ of $\GL(\underline{V})$ on $\underline{V}^{\otimes r}$
\begin{equation}
	\Phi(g)(v_1 \otimes v_2 \otimes \cdots \otimes v_r)=g(v_1) \otimes g(v_2) \otimes \cdots \otimes g(v_r).
\end{equation}
for any $g \in \GL(\underline{V})$ and any monomial tensor product $v_1 \otimes v_2 \otimes \cdots \otimes v_r \in \underline{V}^{\otimes r}$.

In the meanwhile, $(\underline{V}^{\otimes r},\Psi)$ naturally becomes a representation of $\frak S_r$ with following permutation action
\begin{equation}
	\Psi(\sigma)(v_1 \otimes v_2 \otimes \cdots \otimes v_r)=v_{\sigma^{-1}(1)} \otimes v_{\sigma^{-1}(2)} \otimes \cdots \otimes v_{\sigma^{-1}(r)}.
\end{equation}
for any $\sigma \in \frak S_r$ and any monomial tensor product $v_1 \otimes v_2 \otimes \cdots \otimes v_r \in \underline{V}^{\otimes r}$.

The classical Schur-Weyl duality implies that the images of $\Phi$ and $\Psi$ are double centralizers in $\text{End}_{\mathbb{C}}(\underline{V}^{\otimes r})$, namely
\begin{equation}
	\text{End}_{\GL(\underline{V})}(V^{\otimes r})=\mathbb{C}\Psi(\frak S_r),
	\text{End}_{\frak S_r}(\underline{V}^{\otimes r})=\mathbb{C}\Phi(\frak \GL(\underline{V})).
\end{equation}
Here $\mathbb{C}\Psi(\frak S_r)$ and $\mathbb{C}\Phi(\frak \GL(\underline{V}))$ are the subalgebras of $\text{End}_{\mathbb{C}}(\underline{V}^{\otimes r})$ that generated by $\Psi(\frak S_r)$ and $\Phi(\frak \GL(\underline{V}))$ respectively.

\subsection{Degenerate double Hecke algebras.} From now on, we always assume that $G=\GL(V) \cong \GL_n$.

For given positive integers $r$ and $l$ with $r>l$, we define the $l$th degerate double Hecke algebra $\mathcal{H}^l_r$ of $\frak S_r$, which is generated by $\{{\rm{\textbf{x}}}_\sigma \mid \sigma \in \frak S_l\}$ and $\{\rm{\textbf{s}}_\textit{i} \mid \textit{i}=1,2,\cdots,\textit{r}-1\}$ with relations as follows.
\begin{align}\label{DDHA1}
\rm{\textbf{s}}_\textit{i}^2=1, \rm{\textbf{s}}_\textit{i}\rm{\textbf{s}}_\textit{j}=\rm{\textbf{s}}_\textit{j}\rm{\textbf{s}}_\textit{i} \text{  for } 0<\textit{i} \neq \textit{j}\leq \textit{r}-1, |\textit{j}-\textit{i}|>1;
\end{align}
\begin{align}\label{DDHA2}
 \rm{\textbf{s}}_\textit{i}\rm{\textbf{s}}_\textit{j}\rm{\textbf{s}}_\textit{i}=\rm{\textbf{s}}_\textit{j}\rm{\textbf{s}}_\textit{i}\rm{\textbf{s}}_\textit{j} \text{  for } 0<\textit{i} \neq \textit{j}\leq \textit{r}-1, |\textit{j}-\textit{i}|=1;
\end{align}
\begin{align}\label{DDHA3}
{\rm{\textbf{x}}}_\sigma{\rm{\textbf{x}}}_\mu={\rm{\textbf{x}}}_{\sigma \circ \mu} \text{ for } \sigma,\mu \in \frak S_l;
\end{align}
\begin{align}\label{DDHA4}
\rm{\textbf{s}}_\textit{i}{\rm{\textbf{x}}}_\sigma={\rm{\textbf{x}}}_{\rm{\textbf{s}}_\textit{i} \circ \sigma},{\rm{\textbf{x}}}_\sigma\rm{\textbf{s}}_\textit{i}={\rm{\textbf{x}}}_{\sigma \circ \rm{\textbf{s}}_\textit{i}} \text{ for } \sigma \in \frak S_\textit{\textit{l}},\textit{i}<\textit{l};
\end{align}
\begin{align}\label{DDHA5}
\rm{\textbf{s}}_\textit{i}{\rm{\textbf{x}}}_\sigma={\rm{\textbf{x}}}_\sigma={\rm{\textbf{x}}}_\sigma\rm{\textbf{s}}_\textit{i} \text{ for } \sigma \in \frak S_\textit{\textit{l}},\textit{i}>\textit{l}.
\end{align}
This is an infinite-dimensional associative algebra.

The degenerate double Hecke algebra $\mathcal{H}_r$ of $\frak S_r$ is an associative algebra with generators $\rm{\textbf{s}}_\textit{i}$ ($i=1,2,\cdots,r-1$) and ${\rm{\textbf{x}}}_\sigma^{(l)}$ ($\sigma \in \frak S_r, l=0,1,\cdots,r$) with relations as (\ref{DDHA1})-(\ref{DDHA5}) in which ${\rm{\textbf{x}}}_\sigma,{\rm{\textbf{x}}}_\mu$ are replaced by ${\rm{\textbf{x}}}_\sigma^{(l)},{\rm{\textbf{x}}}_\mu^{(l)}$, and additional ones:
\begin{align}
{\rm{\textbf{x}}}_\delta^{(l)}{\rm{\textbf{x}}}_\gamma^{(k)}=0 \text{ for } \delta \in \frak S_l, \gamma \in \frak S_k \text{ with } l\neq k.
\end{align}

For any $m \in \mathbb{Z}_{>0}$, we define $\underline{m}:=\{1,2,\cdots,m\}$. Suppose $V=\bigoplus_{i=1}^n \mathbb{C}\eta_i$, then $\{\eta_1,\cdots,\eta_n,\eta\}$ is a basis of $\underline{V}$. Clearly, all $\eta_\textbf{i}:=\eta_{i_1} \otimes \cdots \otimes \eta_{i_r}$ for $\textbf{i}=\{i_1,\cdots,i_r\} \in \underline{(n+1)}^r$ form a basis of $\underline{V}^{\otimes r}$ where $\eta_{n+1}:=\eta$. For any $I \subseteq \underline{r}$, if $i \in I$, we set $V_i:=V$, otherwise $V_i:=\mathbb{C}\eta$. Now we define a subspace $\underline{V}_I^{\otimes r}$ of $\underline{V}^{\otimes r}$ as $\underline{V}_I^{\otimes r}:=V_1\otimes \cdots \otimes V_r$. In particular, $\underline{V}_\phi^{\otimes r}=\mathbb{C}\eta\otimes \cdots \otimes \mathbb{C}\eta$ with $r$ copies. Suppose $I=\{i_1,\cdots,i_l\} $ where $i_1<i_2<\cdots<i_l$, it is not hard to see that $\underline{V}_I^{\otimes r}$ has a basis as follows 
\begin{align*}
\{\eta_\textbf{j}:=\eta_{j_1} \otimes \cdots \otimes \eta_{j_r} \mid j_{i_k} \in \underline{n},k=1,\cdots,l;j_{d}=n+1 \text{ for } d \neq i_k\}.
\end{align*}
Let $\underline{V}_l^{\otimes r}:=\bigoplus\limits_{I \subseteq \underline{r},~ \sharp I=l}\underline{V}_I^{\otimes r}$, thus $\underline{V}^{\otimes r}$ has a decomposition as a linear space $\underline{V}^{\otimes r}=\bigoplus_{l=0}^r \underline{V}_l^{\otimes r}$.

Given $\sigma \in \frak S_l$ and $I=\{i_1,\cdots,i_l\} $ where $i_1<i_2<\cdots<i_l$, assume that $v_1 \otimes \cdots \otimes v_r \in \underline{V}_I^{\otimes r}$. Let $w_{i_k}:=v_{i_{\sigma^{-1}(k)}}$ with $1 \leq k \leq l$, while for $j \notin I$, set $w_j:=v_j$. We can define $x_\sigma \in \text{End}_\mathbb{C}(\underline{V}_I^{\otimes r})$ with
\begin{align*}
x_\sigma(v_1 \otimes \cdots \otimes v_r)=w_1 \otimes \cdots \otimes w_r.
\end{align*}
Now we extend $x_\sigma$ to an element $x_\sigma^{I}$ of $\text{End}_\mathbb{C}(\underline{V}_l^{\otimes r})$ by annihilating any other summand $\underline{V}_J^{\otimes r}$ with $J \neq I$ and $\sharp J=l$.

\cite[Lemma 3.3]{BYY} gives a representation of the degenerate double Hecke algebra $\mathcal{H}_r$ on $\underline{V}^{\otimes r}$.

\begin{lemma}\label{rep}
	There is a representation $\Xi:\mathcal{H}_r \rightarrow \text{End}_\mathbb{C}(\underline{V}^{\otimes r})$ defined via:
	\begin{itemize}
		\item[(1)] $\Xi|_{\mathbb{C}\frak S_r}=\Psi$;
		\item[(2)] For any ${\rm{\mathbf{x}}}_\sigma^{(l)} \in \mathcal{H}_r^l$ with $\sigma \in \frak S_l$, $l=0,1,\cdots,r$,  $\Xi({\rm{\mathbf{x}}}_\sigma^{(l)})|_{\underline{V}_k^{\otimes r}}=\left\{
		\begin{aligned}
		0 ,k \neq l;\\
		x_\sigma^{\underline{l}},k=l.
		\end{aligned}
		\right.$ 
	\end{itemize}
\end{lemma}

\begin{remark}
	Lemma \ref{rep} implies that $\Xi(\tau)(\underline{V}_l^{\otimes r}) \subseteq \underline{V}_l^{\otimes r}$ for all $\tau \in \mathcal{H}_r^l$, so we can give a representation of $\mathcal{H}^l_r$ on $\underline{V}_l^{\otimes r}$.
	
	By (1), for $\rm{\mathbf{s}}_\textit{i} \in \mathcal{H}_r$ with $i=1,2,\cdots,r-1$, we have $\Xi(\mathbf{s}_\textit{i})=\Psi(\mathbf{s}_\textit{i})$.
\end{remark}

For $0 \leq l \leq r$, we define a representation $\Xi^{(l)}:\mathcal{H}_r^l \rightarrow \text{End}_\mathbb{C}(\underline{V}^{\otimes r})$ via:
\begin{align}
\text{For } \tau \in \mathbb{C}\frak S_r, \text{ set } \Xi^{(l)}(\tau)|_{\underline{V}_k^{\otimes r}}:=\left\{
\begin{aligned}
\Psi(\tau)|_{\underline{V}_l^{\otimes r}}, \text{ } k=l;\\
0, \text{ } k \neq l.
\end{aligned} \text{ }
\right.
\end{align}
\begin{align}
\text{For any } {\rm{\mathbf{x}}}_\sigma^{(l)} \in \mathcal{H}_r^l \text{ with } \sigma \in \frak S_l, \text{ set } \Xi^{(l)}({\rm{\mathbf{x}}}_\sigma^{(l)}):=\Xi({\rm{\mathbf{x}}}_\sigma^{(l)}).
\end{align}

Clearly, $\Xi^{(l)}$ can be viewed as a representation on $\underline{V}_l^{\otimes r}$, i.e. homomorphism $\Xi^{(l)}:\mathcal{H}_r^l \rightarrow \text{End}_\mathbb{C}(\underline{V}_l^{\otimes r})$.

\subsection{Levi Schur-Weyl duality} Set $D(n,r):=\Xi(\mathcal{H}_r)$, which is called the finite-dimensional degenerate double Hecke algebra (DDHA for short) of $\frak S_r$. We should note that $\mathcal{H}_r$ is infinite-dimensional.

Let $D(n,r)_l:=\Xi^{(l)}(\mathcal{H}^l_r)$ for $l=0,1,\cdots,r$. By \cite[Lemma 3.5(1)]{BYY}, we have a decomposition of $D(n,r)$ as a vector space
\begin{align}\label{sum}
D(n,r)=\bigoplus_{l=0}^r D(n,r)_l.
\end{align}

Now we list the Levi Schur-Weyl duality and Parabolic Schur-Weyl duality in \cite{BYY}.

\begin{thm}(\cite[Theorem 4.3]{BYY})
The following duality holds:
\begin{align}
\text{End}_{\GL_n \times \textbf{G}_m}(\underline{V}^{\otimes r})=D(n,r);
\end{align}
\begin{align}
\text{End}_{D(n,r)}(\underline{V}^{\otimes r})=\mathbb{C}\Psi(\GL_n \times \textbf{G}_m).
\end{align}
\end{thm}

\begin{thm}(\cite[Theorem 5.3]{BYY})\label{A parabolic}
	Set $D(n,r)^{V}=\{\phi \in D(n,r) \mid \Phi(e^v) \circ \phi=\phi \circ \Phi(e^v), \forall v \in V\}$.
	Then the following holds:
	\begin{align}
	\text{End}_{\underline{\GL_n} \rtimes \textbf{G}_m}(\underline{V}^{\otimes r})=D(n,r)^{V}.
	\end{align}
\end{thm}

\section{Tensor invariant of enhanced reductive group $\textbf{G}_m \ltimes \underline{\GL(V)}$.}\label{2}
Keep the notations as in Section \ref{1}, and in this section, we assume that $G=\GL(V) \cong \GL_n$. For the symmetric group $\frak S_r$, we denote by $(i,i+1)$ for
$i=1,\cdots, r-1$, the transposition just interchanging $i$ and $i+1$, and fixing the others.

Given $\sigma \in \mathbb{C}\Psi(\frak S_r)$ and $I \subseteq \underline{r}$, define $\sigma^{[I]} \in \text{End}_\mathbb{C}(\underline{V}^{\otimes r})$ as below:
\begin{align}\label{def}
 \sigma^{[I]}|_{\underline{V}_J^{\otimes r}}:=\left\{
 \begin{aligned}
 \sigma|_{\underline{V}_I^{\otimes r}},  \text{ } I=J; \\
0, \text{ }, I \neq J.
 \end{aligned}
 \right.
\end{align}
By definition, $\sigma^{[I]}(\underline{V}_l^{\otimes r}) \subseteq \underline{V}_l^{\otimes r}$.
For $I=\{i_1, \cdots i_l\}$ and $J=\{j_1, \cdots j_l\}$ with $i_1 < \cdots <i_l$, $j_1 < \cdots < j_l$ and $0 \leq l \leq r$, there exists $\varepsilon_{J,I} \in \frak S_r$ satisfying $\varepsilon_{J,I}(i_t)=j_t$ for all $1 \leq t \leq l$. Now we define $E_{J,I}:=(\Psi(\varepsilon_{J,I}))^{[I]}$, and note that $E_{J,I}=\text{id}^{[I]}$ when $l=0$ (i.e. $I=J=\phi$). Although the element $\varepsilon_{J,I}$ in $\frak S_r$ satisfying $\varepsilon_{J,I}(i_t)=j_t$ ($\forall 1 \leq t \leq l$) is not unique, $(\Psi(\varepsilon_{J,I}))^{[I]}$ does not depend on the choice of $\varepsilon_{J,I}$. It is easy to see that $E_{J,I}(\underline{V}_I^{\otimes r}) \subseteq \underline{V}_J^{\otimes r}$. Let $D_{[I]}:=\{\sigma^{[I]} \mid \sigma \in \mathbb{C} \Psi(\frak S_r)\}$, it is clear that $\sum\limits_{I \subseteq \underline{r},~\sharp I=l}D_{[I]}$ is a direct sum, i.e. $\bigoplus\limits_{I \subseteq \underline{r},~\sharp I=l}D_{[I]}$. Now we suppose $n \geq r$.

\begin{lemma}\label{structure}
	$D(n,r)_l=\bigoplus\limits_{I \subseteq \underline{r},~\sharp I=l}D_{[I]}$. 
\end{lemma}
\begin{proof}
	Since $n \geq r$, \cite[Remark 3.6]{BYY} implies that $D(n,r)_l$ has a basis as follows 
	\begin{align*}
	\{E_{J,I} \circ x_{\sigma}^I \mid I,J \subseteq \underline{r}, \sharp I=\sharp J=l, \sigma \in \frak S_l\}.
	\end{align*} 

	For $x_{\sigma}^I$ with $\sharp I=l$, there exists $\sigma' \in \frak S_r$ such that $(\Psi(\sigma'))^{[I]}=x_{\sigma}^I$. Thus $E_{J,I} x_{\sigma}^I=(\Psi(\varepsilon_{J,I}))^{[I]} \cdot (\Psi(\sigma'))^{[I]}=(\Psi(\varepsilon_{J,I} \cdot \sigma'))^{[I]} \in D_{[I]}$, which yields that $D(n,r)_l \subseteq \bigoplus\limits_{I \subseteq \underline{r},~\sharp I=l}D_{[I]}$.
	
	Conversely, if $\tau \in \frak S_r$, in order to conclude $D_{[I]} \subseteq D(n,r)_l$ with $\sharp I=l$, it suffices to show $(\Psi(\tau))^{[I]} \in D(n,r)_l$. Suppose $I=\{i_1, \cdots, i_l\}$ with $i_1 < \cdots <i_l$ and $\tau(I)=\{j_1, \cdots, j_l\}$ with $j_1 < \cdots < j_l$. Moreover, assume that $m_1, \cdots, m_l$ are different numbers in $\underline{l}$ satisfying $\tau(i_t)=j_{m_t}$ ($\forall t \in \underline{l}$). There exist $\sigma' \in \frak S_r$ satisfying $\sigma'(i_t)=i_{m_t}$ ($\forall 1 \leq t \leq l$) and set $\sigma=\begin{pmatrix}
	1 & \cdots & l \\
	m_1 & \cdots & m_l
	\end{pmatrix} \in \frak S_l$. Recall that $\varepsilon_{\tau(I),I}(i_t)=j_t$ for all $1 \leq t \leq l$ (i.e. $\varepsilon_{\tau(I),I}(i_{m_t})=j_{m_t}$ for all $1 \leq t \leq l$).
	 It is not hard to see that $(\Psi(\tau))^{[I]}=(\Psi(\varepsilon_{\tau(I),I} \cdot \sigma'))^{[I]}=(\Psi(\varepsilon_{\tau(I),I}))^{[I]} \cdot (\Psi(\sigma'))^{[I]}=E_{\tau(I),I} \cdot x_\sigma^I \in D(n,r)_l$.
	 
	 For all $I \subseteq \underline{r}$ with $\sharp I=l$, we have $D_{[I]} \subseteq D(n,r)_l$. Hence $\bigoplus\limits_{I \subseteq \underline{r},~\sharp I=l}D_{[I]} \subseteq D(n,r)_l$.
\end{proof}
\begin{remark}\label{decom}
Lemma \ref{structure} and (\ref{sum}) implies that $D(n,r)=\bigoplus\limits_{I \subseteq \underline{r}}D_{[I]}$.
\end{remark}

Keep in mind that $n \geq r$. Fix a basis $\{v_1,\cdots,v_n\}$ of $V$, and recall that $\underline{V}=V \oplus \mathbb{C}\eta$. Given $w \in V$ and $J \subseteq \underline{r}$, we define $a_i \in \underline{V}$ ($1 \leq i \leq r$) and $b_i \in \underline{V}$ ($1 \leq i \leq r$) as follows:
\begin{align*}
\text{If } i \in J, \text{ set } a_i:=w+\eta; \text{ if } i \notin J, \text{ set } a_i:=v_i.
\end{align*}                      
\begin{align*}
\text{If } i \in J, \text{ set }b_i:=\eta; \text{ if } i \notin J, \text{ set } b_i:=v_i.
\end{align*}
Let $A_w^J:=a_1 \otimes \cdots \otimes a_r-b_1 \otimes \cdots \otimes b_r \in \underline{V}^{\otimes r}$, we have the following Lemma.

\begin{lemma}\label{key lemma}
	Suppose $n \geq r$. For $\delta \in \mathbb{C} \Psi(\frak S_r)$ and $\phi \neq J \subseteq \underline{r}$, if $\forall w \in V$, we always have $\delta(A_w^J)=0$, then $\delta(\underline{V}_{\underline{r} \backslash J}^{\otimes r})=0$.
\end{lemma}

\begin{proof}
	We assume that $\sharp J=r-k$ (i.e. $\sharp (I \backslash J)=k$). Obviously, it is enough for us to prove the case $J=\underline{r} \backslash \underline{k}=\{k+1,\cdots,r\}$.
	
	Now we supoose $J=\{k+1,\cdots,r\}$ and take $w=v_{k+1}+\cdots+v_r$, then $A_w^J=v_1 \otimes \cdots \otimes v_k \otimes (v_{k+1}+\cdots+v_r+\eta)^{\otimes (r-k)}-v_1 \otimes \cdots \otimes v_k \otimes \eta^{\otimes (r-k)}$. Set $G_1:=\langle(i,i+1) \mid 1\leq i \leq k-1\rangle$ and $G_2:=\langle(i,i+1) \mid k+1\leq i \leq r-1\rangle$, then $G_1 \cong \frak S_k$ and $G_2 \cong \frak S_{r-k}$. For all $s \in \frak S_r$, we define $\underline{s}:=\{s(x) \mid x \in \underline{k} \text{ and } k+1 \leq s(x) \leq r\}$ and $\bar{s}:=\{s(x) \mid x \in \underline{r}\backslash\underline{k} \text{ and } 1\leq s(x) \leq k\}$. It is clear that $\underline{s}=(\underline{r}\backslash\underline{k})\backslash\{{s(x) \mid x \in \underline{r}\backslash\underline{k} \text{ and } s(x) \in \underline{r}\backslash\underline{k}}\}$, $\bar{s}=\underline{k} \backslash \{s(x) \mid x \in \underline{k} \text{ and } s(x) \in \underline{k}\}$ and $\sharp \underline{s}=\sharp \bar{s}$. For $B \subseteq \underline{k}$ and $A \subseteq \underline{r}\backslash\underline{k}=\{k+1,\cdots,r\}$, if $\sharp A=\sharp B$, set $\Psi_{A,B}:=\sum\limits_{s \in \frak S_r \atop \underline{s}=A,~\bar{s}=B}\mathbb{C}\Psi(s)$; if $\sharp A \neq \sharp B$, there does not exist $s \in \frak S_r$ satisfying $\underline{s}=A$ and $\bar{s}=B$, then we set $\Psi_{A,B}:=0$. Obviously, $\mathbb{C} \Psi(\frak S_r)=\sum\limits_{A \subseteq \underline{r}}\sum\limits_{B \subseteq \underline{r}}\Psi_{A,B}=\sum\limits_{A \subseteq \underline{r}\backslash\underline{k},~B \subseteq \underline{k} \atop \sharp A=\sharp B}\Psi_{A,B}$.
	
	Clearly, for $s \in \frak S_r$, $\underline{s}=\phi \Leftrightarrow \bar{s}=\phi \Leftrightarrow s \in G_1G_2$. Given $A \subseteq \underline{r}\backslash\underline{k}$ and $B \in \underline{k}$ with $\sharp A=\sharp B=q \neq 0$ ($\Rightarrow A \neq \phi, B \neq \phi$), assume that $A=\{i_1, \cdots i_q\}$ and $B=\{j_1, \cdots j_q\}$, now we set $\theta=(i_1,j_1) \cdots (i_q,j_q)$. For all $\varepsilon \in \frak S_r$ with $\underline{\varepsilon}=A$ and $\bar{\varepsilon}=B$, we have $\theta \cdot \varepsilon \in G_1 G_2=\langle G_1,G_2 \rangle$.
	
	Hence fix $A \subseteq \underline{r}\backslash\underline{k}$ and $B \subseteq \underline{k}$ with $\sharp A=\sharp B$, $\forall \varepsilon \in \frak S_r$, we have $\theta \cdot \varepsilon \in G_1 G_2$ (Here $\theta:=1=\text{id}$, if $A=B=\phi$). Namely,
	\begin{align}\label{tran}
	\varepsilon=\theta^{-1} \cdot \varepsilon'=\theta \cdot \varepsilon', \text{ where } \varepsilon'=\theta\varepsilon \in G_1G_2, \text{ and note that } \theta=\theta^{-1} \in \frak S_r \backslash (G_1G_2).
	\end{align}
	
	The terms in $A_w^J$ containing $v_1, \cdots ,v_r$ can be written as $\Psi(s)(v_1 \otimes \cdots \otimes v_r)$, where $s$ ranges over $G_2$. Hence $\delta(A_w^J)=0$ implies that $\delta(\sum_{s \in G_2}\Psi(s)(v_1 \otimes \cdots \otimes v_r))=0$. Thanks to $\mathbb{C} \Psi(\frak S_r)=\sum\limits_{A \subseteq \underline{r}}\sum\limits_{B \subseteq \underline{r}}\Psi_{A,B}$ and (\ref{tran}), there exist $m \in \mathbb{Z}_{>0}$, $\phi_1,\cdots,\phi_m \in \frak S_r \backslash (G_1G_2)$ and $T_0,T_1,\cdots,T_m \in \mathbb{C} \Psi(G_1G_2)$, such that $\delta=T_0+\sum_{i=1}^{m}\Psi(\phi_i) \circ T_i=\sum_{i=0}^{m}\Psi(\phi_i) \circ T_i$ (Here $\phi_0=1$), and satisfying 
	\begin{align}\label{Sum}
	\Psi(\phi_i) \circ T_i \in \Psi_{A_i,B_i} \text{ and } \Psi(\phi_i) \in \Psi_{A_i,B_i} \text{, with } 0 \leq i \leq m, A_i \subseteq \underline{r}\backslash\underline{k}, B_i \subseteq \underline{k}, \sharp A_i=\sharp B_i \neq 0.
	\end{align}
	 as well as $A_i \neq A_j$ or $B_i \neq B_j$, $\forall 0 \leq i \neq j \leq m$, i.e.
	\begin{align}\label{diff}
	 (A_i,B_i) \neq (A_j,B_j) \text{ for all } 0 \leq i \neq j \leq m.
	\end{align}
	By (\ref{Sum}), we see that 
	\begin{align}\label{chaifen}
	\Psi(\phi_i) \circ T_i \circ (\sum_{s \in G_2}\Psi(s)) \in \Psi_{A_i,B_i}.
	\end{align} 
	
	We have $(T_0 \circ (\sum_{s \in G_2}\Psi(s)))(v_1\otimes \cdots \otimes v_r)+(\Psi(\phi_1) \circ T_1 \circ (\sum_{s \in G_2}\Psi(s)))(v_1\otimes \cdots \otimes v_r)+\cdots+(\Psi(\phi_m) \circ T_m \circ (\sum_{s \in G_2}\Psi(s)))(v_1\otimes \cdots \otimes v_r)=0$. Due to (\ref{diff}) and (\ref{chaifen}), it yields that 
	\begin{align*}
	(\Psi(\phi_i) \circ T_i \circ (\sum_{s \in G_2}\Psi(s)))(v_1\otimes \cdots \otimes v_r)=0, \text{ } \forall i=0,1,\cdots,m.
	\end{align*}
	Thus we have
	\begin{align*}
	(\Psi(\phi_i^{-1}) \circ \Psi(\phi_i) \circ T_i \circ (\sum_{s \in G_2}\Psi(s)))(v_1\otimes \cdots \otimes v_r)=0, \text{ } \forall i=0,1,\cdots,m.
	\end{align*}
	Namely
	\begin{align}\label{claim}
	 (T_i \circ (\sum_{s \in G_2}\Psi(s)))(v_1\otimes \cdots \otimes v_r)=0, \text{ } \forall i=0,1,\cdots,m.
	\end{align}
	
Since $T_i \in \mathbb{C} \Psi(G_1G_2)$ and $n \geq r$, there exist $l \in \mathbb{N}^*$, $k_j \in \mathbb{C}$, $g_t^{(j)} \in G_t$ with $t=1,2$ and $j \in \underline{l}=\{1,\cdots,l\}$, such that $T_i=\sum_{j=1}^{l}k_j\Psi(g_1^{(j)}) \circ \Psi(g_2^{(j)})$. Hence 
\begin{align*}
	T_i \circ (\sum_{s \in G_2}\Psi(s))=(\sum_{j=1}^l k_j \Psi(g_1^{(j)}) \circ \Psi(g_2^{(j)})) \circ (\sum_{s \in G_2}\Psi(s))=\sum_{j=1}^l \sum_{s \in G_2} k_j \Psi(g_1^{(j)}) \circ \Psi(g_2^{(j)}) \circ \Psi(s)
\end{align*}
\begin{align*}
=\sum_{j=1}^l  k_j \Psi(g_1^{(j)}) (\sum_{s \in G_2} \Psi(g_2^{(j)}\cdot s))=(\sum_{j=1}^l  k_j \Psi(g_1^{(j)}))(\sum_{s \in G_2} \Psi(s)).
\end{align*}
So $(\sum_{j=1}^{l}k_j\Psi(g_1^{(j)})) \circ (\sum_{s \in G_2}\Psi(s))(v_1\otimes \cdots \otimes v_r)=0$. Thanks to $G_1 \cap G_2=\{1\}$, we can conclude that 
\begin{align}\label{2.5}
(\sum_{j=1}^{l}k_j\Psi(g_1^{(j)}))(v_1\otimes \cdots \otimes v_r)=0.
\end{align}
Since $\sum_{j=1}^{l}k_j\Psi(g_1^{(j)}) \in \mathbb{C}\Psi(G_1)$, there exists $h_i \in G_1, c_i \in \mathbb{C}$ with $1 \leq i \leq t$ satisfying $h_p \neq h_q$ for all $1 \leq p<q \leq t$, such that $\sum_{j=1}^{l}k_j\Psi(g_1^{(j)})=\sum_{i=1}^{t}c_i\Psi(h_i)$. By (\ref{2.5}), we see that $\sum_{i=1}^{t}c_i\Psi(h_i)(v_1\otimes \cdots \otimes v_r)=0$. Due to $h_1, \cdots, h_t$ are all different, $\Psi(h_1)(v_1\otimes \cdots \otimes v_r), \cdots, \Psi(h_t)(v_1\otimes \cdots \otimes v_r)$ are linear independent vectors, which implies that $c_1=\cdots=c_t=0$. Hence 
\begin{align}\label{0}
\sum_{j=1}^{l}k_j\Psi(g_1^{(j)})=0.
\end{align}
Note that $\underline{V}_{\underline{k}}^{\otimes r}=V^{\otimes k} \otimes (\mathbb{C}\eta)^{\otimes (r-k)}$, thus $\Psi(g_2^{(j)})|_{\underline{V}_{\underline{k}}^{\otimes r}}=\text{id}_{\underline{V}_{\underline{k}}^{\otimes r}}$. This yields that
\begin{align*}
T_i|_{\underline{V}_{\underline{k}}^{\otimes r}}=\sum_{j=1}^{l}(k_j\Psi(g_1^{(j)}) \circ \Psi(g_2^{(j)}))|_{\underline{V}_{\underline{k}}^{\otimes r}}=\sum_{j=1}^{l} k_j\Psi(g_1^{(j)})|_{\underline{V}_{\underline{k}}^{\otimes r}}\xlongequal{\text{By } (\ref{0})}0.
\end{align*}
Thus $T_i|_{\underline{V}_{\underline{k}}^{\otimes r}}=0$ is proved.
\end{proof}

Now we begin to prove \cite[Conjecture 5.4]{BYY}.

\begin{thm}\label{conjecture}(\cite[Conjecture 5.4]{BYY})
If $n \geq r$, we have $D(n,r)^V=\mathbb{C}\Psi(\frak S_r)$, i.e. $\text{End}_{\underline{G} \rtimes \textbf{G}_m}(\underline{V}^{\otimes r})=\mathbb{C}\Psi(\frak S_r)$.
\end{thm}

\begin{proof}
	Clearly, Schur-Weyl duality implies that $\mathbb{C}\Psi(\frak S_r)=\text{End}_{\GL(\underline{V})}(\underline{V}^{\otimes r}) \subseteq \text{End}_{\underline{G} \rtimes \textbf{G}_m}(\underline{V}^{\otimes r})$, i.e.
	$\mathbb{C}\Psi(\frak S_r) \subseteq D(n,r)^V$. Now we just need to show that $D(n,r)^V \subseteq \mathbb{C}\Psi(\frak S_r)$. 
	
	$\forall \sigma \in D(n,r)^V$, Remark \ref{decom} implies that there exist $\tau_I \in \mathbb{C}\Psi(\frak S_r)$ with all $I \subseteq \underline{r}$, such that $\sigma=\sum_{I \subseteq \underline{r}} \tau_I^{[I]}$ (Note that $\tau_I^{[I]}$ is defined in (\ref{def})). Now we use induction to prove.
	
	If $\sharp I=r$, then it must have $I=\underline{r}$. So we assume that there exist $\tau \in \mathbb{C}\Psi(\frak S_r)$ with $0 \leq k \leq r-1$, such that $\forall I \subseteq \underline{r}$ with $\sharp I \geq k+1$ we have $\tau_I^{[I]}=\tau^{[I]}$ ($\tau^{[I]}$ is defined in (\ref{def})). Then $\sigma=\sum\limits_{I \subseteq \underline{r}, ~ \sharp I \leq k} \tau_I^{[I]}+\sum\limits_{I \subseteq \underline{r}, ~ \sharp I >k} \tau^{[I]}$. Recall that $\{v_1, \cdots, v_n\}$ is a basis of $V$ and $\underline{V}=V \oplus \mathbb{C}\eta$, it is easy to see $\forall w \in V,e^w(\eta)=w+\eta$ and $e^w|_{V}=\text{id}_V$.
	
	For $w \in V$, we have $v_1 \otimes \cdots \otimes v_k \otimes (w+\eta)^{\otimes (r-k)}=v_1 \otimes \cdots \otimes v_k \otimes \eta^{\otimes (r-k)}+A_w^{\underline{r} \backslash \underline{k}}$, where $A_w^{\underline{r} \backslash \underline{k}}=A_w^{\{k+1,\cdots,r\}}$ is defined as in Lemma \ref{key lemma}.
	\begin{align*}
	(\sigma \circ \Phi(e^w))(v_1 \otimes \cdots \otimes v_k \otimes \eta^{\otimes (r-k)})
	=\sigma(v_1 \otimes \cdots \otimes v_k \otimes (w+\eta)^{\otimes (r-k)})
	\end{align*}
	\begin{align*}
	=(\sum\limits_{I \subseteq \underline{r}, ~ \sharp I \leq k} \tau_I^{[I]}+\sum\limits_{I \subseteq \underline{r}, ~ \sharp I >k} \tau^{[I]})(v_1 \otimes \cdots \otimes v_k \otimes (w+\eta)^{\otimes (r-k)})
	\end{align*}
	\begin{align*}
	=(\tau_{\underline{k}}^{[\underline{k}]}+\sum\limits_{I \subseteq \underline{r}, ~ \sharp I >k} \tau^{[I]})(v_1 \otimes \cdots \otimes v_k \otimes (w+\eta)^{\otimes (r-k)})
	\end{align*}
	\begin{align*}
	=\tau_{\underline{k}}^{[\underline{k}]}(v_1 \otimes \cdots \otimes v_k \otimes \eta^{\otimes (r-k)})+\tau(A_w^{\underline{r} \backslash \underline{k}})
	\end{align*}
	\begin{align}\label{eq1}
	=\tau_{\underline{k}}(v_1 \otimes \cdots \otimes v_k \otimes \eta^{\otimes (r-k)})+\tau(A_w^{\underline{r} \backslash \underline{k}}).
	\end{align}
	\begin{align*}
	(\Phi(e^w) \circ \sigma)(v_1 \otimes \cdots \otimes v_k \otimes \eta^{\otimes (r-k)})
	=\Phi(e^w)(\tau_{\underline{k}}^{[\underline{k}]}(v_1 \otimes \cdots \otimes v_k \otimes \eta^{\otimes (r-k)}))
	\end{align*}
	\begin{align*}
	=\Phi(e^w)(\tau_{\underline{k}}(v_1 \otimes \cdots \otimes v_k \otimes \eta^{\otimes (r-k)}))
	\end{align*}
	\begin{align*}
	=\tau_{\underline{k}}(\Phi(e^w)(v_1 \otimes \cdots \otimes v_k \otimes \eta^{\otimes (r-k)})) \text{ }(\text{By Schur-Weyl duality})
	\end{align*}
	\begin{align*}
	=\tau_{\underline{k}}(v_1 \otimes \cdots \otimes v_k \otimes (w+\eta)^{\otimes (r-k)}) 
	\end{align*}
	\begin{align}\label{eq2}
	=\tau_{\underline{k}}(v_1 \otimes \cdots \otimes v_k \otimes \eta^{\otimes (r-k)})+\tau_{\underline{k}}(A_w^{\underline{r} \backslash \underline{k}}).
	\end{align}
	
Since $\sigma \in D(n,r)^V$, i.e. $\Phi(e^w) \circ \sigma=\sigma \circ \Phi(e^w)$, ($\ref{eq1}$) and ($\ref{eq2}$) yields that $\forall w \in V$, $\tau(A_w^{\underline{r} \backslash \underline{k}})=\tau_{\underline{k}}(A_w^{\underline{r} \backslash \underline{k}})$. Similarly, $\forall w \in V$ and $\forall J \subseteq \underline{r}$ with $\sharp J=r-k$, we have $\tau(A_w^{J})=\tau_{\underline{r} \backslash J}(A_w^{J})$, namely $(\tau-\tau_{\underline{r} \backslash J})(A_w^{J})=0$. Due to Lemma $\ref{key lemma}$ (note that $k \leq r-1$), we conclude that $(\tau-\tau_{\underline{r} \backslash J})(\underline{V}_{\underline{r}\backslash J}^{\otimes r})=0$. Thus $\forall I \subseteq \underline{r}$ with $\sharp I=k$, we have $(\tau-\tau_I)(\underline{V}_I^{\otimes r})=0$, i.e. $\tau|_{\underline{V}_I^{\otimes r}}=\tau_I|_{\underline{V}_I^{\otimes r}}$, which implies that $\tau^{[I]}=\tau_I^{[I]}$. By induction, we have $\tau^{[I]}=\tau_I^{[I]}$ for all $I \subseteq \underline{r}$. Thus $\sigma=\sum\limits_{I \subseteq \underline{r}} \tau_I^{[I]}=\tau \in \mathbb{C} \Psi(\frak S_r)$, i.e. $D(n,r)^V \subseteq \mathbb{C}\Psi(\frak S_r)$.
\end{proof}

\section{Enhanced tensor invariants.}\label{3}
Keep the notations as in \S\ref{1}. In this section, we still set $G=GL(V) \cong \GL_n$, and identify $\underline{V}^{* \otimes r}$ with $(\underline{V}^{\otimes r})^*$. For convenience, we denote $\underline{n+1}$ by $\mathcal{N}$. Then there is a natural $\underline{G} \rtimes \textbf{G}_m$-module isomorphism $T:\underline{V}^{\otimes r} \otimes \underline{V}^{* \otimes r} \rightarrow \text{End}_\mathbb{C}(\underline{V}^{\otimes r})$. Due to the parabolic Schur-Weyl duality (i.e. Theorem \ref{conjecture}), it is time to describe the enhanced tensor invariants.

Now we still let $\eta_1, \cdots, \eta_n$ be a basis of $V$ and $\underline{V}=V \oplus \mathbb{C} \eta_{n+1}$ with $\eta_{n+1}:=\eta$. Then $\{\eta_i \mid i=1,2,\cdots, n+1\}$ constitute a basis of $\underline{V}$. Let $\{\eta_i^* \mid i=1,2,\cdots, n+1\}$ be the dual basis of $\underline{V}^*$. Given a multi-index $\textbf{i}=(i_1, \cdots, i_r) \in \mathcal{N}^r$, we still set $\eta_\textbf{i}:=\eta_{i_1} \otimes \cdots \otimes \eta_{i_r}$ and $\eta_\textbf{i}^*:=\eta_{i_1}^* \otimes \cdots \otimes \eta_{i_r}^*$. Then $\eta_\textbf{i}$ and $\eta_\textbf{i}^*$ with $\textbf{i}$ ranges over $\mathcal{N}^r$ form a basis of $\underline{V}^{\otimes r}$ and $\underline{V}^{* \otimes r}$ respectively. Moreover, $\{\eta_\textbf{i} \otimes \eta_\textbf{j}^* \mid (\textbf{i},\textbf{j}) \in \mathcal{N}^r \times \mathcal{N}^r\}$ is a basis of $\underline{V}^{\otimes r} \otimes \underline{V}^{* \otimes r}$.

For $\sigma \in \frak S_r$, define a mixed tensor $C_\sigma$ as follows 
\begin{align}\label{C}
C_\sigma=\sum_{\textbf{i} \in \mathcal{N}^r}\eta_{\sigma.\textbf{i}} \otimes \eta_\textbf{i}^*.
\end{align}

\begin{proposition}\label{A inv} (\cite[Proposition 5.5]{BYY})
	When $D(n,r)^V=\mathbb{C}\Psi(\frak S_r)$, then the space of $\underline{G} \rtimes \textbf{G}_m$-invariants in the mixed-tensor product $\underline{V}^{\otimes r} \otimes \underline{V}^{* \otimes r}$ is generated by $C_\sigma$ defined in (\ref{C}) with $\sigma \in \frak S_r$.
\end{proposition}

Obviously, Theorem \ref{conjecture} and Propsition \ref{A inv} yield the following proposition.

\begin{proposition}
	If $n \geq r$, then the space of $\underline{G} \rtimes \textbf{G}_m$-invariants in the mixed-tensor product $\underline{V}^{\otimes r} \otimes \underline{V}^{* \otimes r}$ is generated by $C_\sigma$ with $\sigma \in \frak S_r$.
\end{proposition}

\section*{Acknowledgment}
The author thanks Bin Shu for his guidance, and thanks Yufeng Yao for his helpful comments and discussions.

\end{document}